\documentclass[10pt,electronic]{amsart}        
\usepackage[swedish, english]{babel}
\usepackage{baskervald}
\usepackage[latin1]{inputenc}
\usepackage{accents}
\usepackage[font=small]{caption}
\usepackage[pdfusetitle]{hyperref}
\hypersetup{
	colorlinks,
	linkcolor={red!50!black},
	citecolor={blue!50!black},
	urlcolor={blue!80!black}
}

\usepackage[hmargin=3.7cm,vmargin=3.5cm]{geometry} 
\usepackage{adjustbox}
\usepackage{url}
\usepackage{graphicx}                         
\usepackage{amsmath}
\usepackage{amsthm}                      
\usepackage{amssymb} 
\usepackage{mathrsfs}
\usepackage{stmaryrd}
\usepackage{mathtools}
\usepackage[shortlabels]{enumitem}
\usepackage{tikz-cd}
\usetikzlibrary{babel}

\usepackage{euscript}
\renewcommand{\mathcal}{\EuScript}

\usepackage[noabbrev,capitalize]{cleveref}
\theoremstyle{plain}                                              \newtheorem{thm}{Theorem}[section]
\newtheorem{lem}[thm]{Lemma}
\newtheorem{prop}[thm]{Proposition}
\newtheorem{cor}[thm]{Corollary}

 \newtheorem{thmA}{Theorem}

\theoremstyle{definition}

\newtheorem{rem}[thm]{Remark}
\newtheorem{defn}[thm]{Definition}

\newtheorem{notation}[thm]{Notation}

\title{The handlebody group is a virtual duality group}

\begin{document}

\author{Dan Petersen}
\author{Richard D. Wade}

\begin{abstract}We show that the mapping class group of a handlebody is a virtual duality group, in the sense of Bieri and Eckmann. In positive genus we give a description of the dualising module of any torsion-free, finite-index subgroup of the handlebody mapping class group as the homology of the \emph{complex of non-simple disc systems}. \end{abstract}
 \maketitle
\newcommand{\eps}{\varepsilon}
\newcommand{\R}{\mathbb R}
\newcommand{\Q}{\mathbb Q}
\newcommand{\C}{\mathbb C}
\newcommand{\M}{\mathcal{M}}
\newcommand{\Mbar}{\overline{\mathcal{M}}}
\newcommand{\HM}{\mathcal{HM}}
\newcommand{\T}{\mathcal T}
\newcommand{\TT}{\mathcal T^{\mathrm{tw}}}
\newcommand{\HT}{\mathcal{HT}}
\newcommand{\Y}{\mathcal Y}
\newcommand{\Sp}{\mathrm{Sp}}
\newcommand{\trop}{\mathrm{trop}}
\newcommand{\Z}{\mathbb Z}

\newcommand{\RGB}{{\mathsf{RGB}}}
\newcommand{\dRGB}{\partial\mathsf{RGB}}
\newcommand{\RG}{\mathsf{RG}}
\newcommand{\dRG}{\partial\mathsf{RG}}
\newcommand{\NS}{\mathsf{NS}}
\newcommand{\D}{\mathsf{D}}

\newcommand{\compNS}{\mathbf{NS}}
\newcommand{\compD}{\mathbf{D}}
\newcommand{\Sph}{\mathbf{Sph}}
\newcommand{\CV}{\mathbf{CV}}

\newcommand{\SD}{\mathsf{S}}
\newcommand{\Po}{\mathsf{P}}
\newcommand{\mcg}{\mathrm{Mod}}
\newcommand{\hmod}{\mathrm{Mod}}
\newcommand{\diff}{\mathrm{Diff}}
\newcommand{\FN}{\mathrm{FN}_\mathcal{P}}
\newcommand{\co}{\colon}
\newcommand{\op}{{\mathrm{op}}}
\newcommand{\altjoin}{\oast}
\newcommand{\out}{\mathrm{Out}}
\newcommand{\Ht}{{\mathrm{ht}}}

\tableofcontents

\section{Introduction}

Let $\Gamma$ be a discrete group. Suppose that there exists a closed $d$-manifold $X$ which is a $K(\Gamma,1)$. Then by Poincar\'e duality on $X$ one has isomorphisms
$$ H^k(\Gamma,A) \cong H_{d-k}(\Gamma, A\otimes D)$$
for all $k \in \Z$ and all $\Gamma$-modules $A$, where $D=\mathbb{Z}$ and the $\Gamma$-action on $D$ corresponds to the orientation local system on $X$. The isomorphisms are natural, being implemented by capping with the fundamental class in $H_d(\Gamma,D)$. A group $\Gamma$ is a \emph{Poincar\'e duality group} of \emph{cohomological dimension} $d$ if it admits such a \emph{dualising module} $D$ and fundamental class in $H_d(\Gamma,D)$ inducing isomorphisms as above.

It is natural to relax the condition that $D=\mathbb{Z}$ as an abelian group. Allowing the dualising module $D$ to be arbitrary leads to the more general notion of a \emph{duality group}, as introduced by Bieri--Eckmann \cite{be}. Many groups of interest in topology and geometric group theory are known to be duality groups, at least virtually: in particular, arithmetic groups, mapping class groups of compact surfaces, and (outer) automorphism groups of free groups are all virtual duality groups \cite{borelserre,harer,bestvinafeighn}. 

McCullough \cite{mccullough} showed that the mapping class group $\mcg(M)$ of a Haken three-manifold $M$ is also a virtual duality group. Away from the Haken case, let $V$ be a solid handlebody of genus $g$, with $b$ marked discs and $p$ marked points on its boundary. In the same paper, McCullough showed that $\mcg(V)$ virtually has a finite classifying space. However, he could only show that $\mcg(V)$ is a virtual duality group when $g=2$, using very special properties of the genus two handlebody group. Our first main theorem shows that the mapping class group of any handlebody is indeed a virtual duality group.

\begin{thmA}\label{thm:A}
For $2g -2+ b +p > 0$, the handlebody group $\hmod(V)$ is a virtual duality
group of dimension $d(g, b, p)$, where $d(g, 0 , 0 ) = 4 g - 5 $; $d(g , b , p ) = 4 g + 2 b + p - 4$ when $g > 0$ and $b + p > 0$; and $d(0,b,p)=2b + p - 3$.
\end{thmA}

\begin{rem} Let $S$ be the surface with $p$ punctures and $b$ boundary components obtained by removing the marked points and the interiors of the marked discs from $\partial V$. Then $\mcg(V)$ can be identified with the subgroup of $\mcg(S)$ consisting of diffeomorphisms of $S$ that extend to $V$. Hirose \cite{hirose} showed that $\mcg(V)$ has the same virtual cohomological dimension as $\mcg(S)$. The fact that the virtual cohomological dimension of $\mcg(S)$ is  $d(g,b,p)$ is due to Harer \cite{harer}.
\end{rem}

One reason to care about a group being a virtual duality group is that it provides a powerful tool for studying its cohomology in high degrees, close to the virtual cohomological dimension; this is in a sense complementary to the phenomenon of homological stability. However, this is only feasible if one knows a meaningful description of the dualising module. We refer the reader to the survey of Br\"uck  \cite{bruck-survey} for more about this circle of ideas and many pointers to the literature. 

In the case of arithmetic groups, the dualising module is called the \emph{Steinberg module}. It is given by the unique nontrivial reduced homology group of the \emph{Tits building} associated with the arithmetic group \cite{borelserre}. The analogue of the Steinberg module for the mapping class group was also determined in Harer's work \cite{harer}: the dualising module of $\mcg(S)$ is isomorphic to the  unique nontrivial reduced homology group of the \emph{complex of curves} (as a $\mcg(S)$-module). 

The second main theorem of this paper describes the analogue of the Steinberg module for the handlebody group, i.e.\ the dualising module of $\mcg(V)$. We do this only when $g>0$; indeed, the map $\hmod(V)\to\mcg(S)$ is an isomorphism when $g=0$, and in this case we already have a description of the dualising module from Harer's work. 

\begin{defn}By an \emph{essential disc} in $V$ we shall always mean an isotopy class of embedded disc, with boundary in the interior of $S$, whose boundary does not bound an annulus or a (punctured) disc in $S$. By a \emph{disc system} we mean a set of distinct disjointly embeddable essential discs. A disc system is \emph{simple} if it cuts $V$ into pieces of genus zero, and is \emph{non-simple} otherwise. 

The \emph{disc complex} $\compD=\compD(V)$ has vertex set equal to the essential discs in $V$. A set of discs spans a simplex in $\compD$ if and only if they form a disc system in $V$. We let $\compNS = \compNS(V)$ denote the \emph{complex of non-simple disc systems} in $V$, which is the subcomplex $\compNS \subset \compD$ whose $k$-simplices are given by non-simple systems of $k+1$ essential discs in the handlebody $V$. \end{defn}

\begin{thmA}\label{thm:B}Let $2g -2+ b +p > 0$ and $g>0$. Then $\compNS(V)$ is homology equivalent to a wedge of spheres of dimension $\nu(g,b,p)$, where $\nu(g,0,0) = 2g-3$ and $\nu(g,b,p) = 2g-4+b+p$ if $b+p>0$. 
\end{thmA}

\begin{rem}
    If $g>2$, then $\compNS$ is also simply connected, hence homotopy equivalent to a wedge of spheres. See \cref{prop: simple connectivity}. Although we believe that $\compNS$ is always homotopic to a wedge of spheres, we were unable to show this in general.
\end{rem}

\begin{thmA}\label{thm:C} Keep notation as in \cref{thm:B}. 
    The dualising module of any torsion-free finite index subgroup of $\hmod(V)$ is $\widetilde H_{\nu(g,b,p)}(\compNS(V),\Z)$.
\end{thmA}

\begin{rem}
    The preceding theorem is correct also in the exceptional case $g=1$, $b+p=1$, if one uses the convention that the reduced homology of the empty space is $\Z$ in degree $-1$. 
\end{rem}
\begin{rem}\label{rem: connection to outer space}
The appearance of the simplicial complex $\compNS$ is striking, especially in relation to Culler--Vogtmann's \emph{Outer space} \cite{cv}. Let us expand on this connection, assuming for simplicity that $V$ has genus $g\geq2$ and $b=p=0$.  A disc in $V$ gives rise to a sphere in the double $M$ of $V$, which is the connected sum of $g$ copies of $S^1 \times S^2$, and $\pi_1(V)\cong\pi_1(M)\cong F_g$.  The action of $\mcg(V)$ on $\pi_1(V)$ induces a surjective homomorphism $\hmod(V) \to \out(F_g)$.

Doubling discs induces a natural map $f:\compD \to \Sph$ to the \emph{sphere complex} of $M$, which is equivariant under the map $\hmod(V) \to \out(F_g)$, and also surjective. In the same manner as above, one defines a sphere system to be \emph{simple} if it cuts the doubled handlebody into simply connected pieces. We denote the subcomplex of non-simple sphere systems by $\partial \CV \subset \Sph$, so that $f^{-1}(\partial \CV) = \compNS$. The complement $\Sph \setminus \partial \CV$ is precisely Outer space. Thus $\compNS$ is a close analogue of the simplicial boundary of Outer space.  

Theorem~\ref{thm:B} gives the homotopy type of $\compNS$ and \cref{thm:C} that its homology is the dualising module of $\hmod(V)$. We do not know the homotopy type of $\partial \CV$, nor do we know its relationship with the dualising module of $\out(F_g)$ (this is explored further in \cite{WW}). Br\"uck and Gupta \cite[Theorem 2]{bruck-gupta} have shown that $\partial \CV$  is homotopy equivalent to the realization of the poset of \emph{free factor systems} of rank $g$, and that it is at least $(g-2)$-connected. Vogtmann \cite{vogtmann} has shown that a reduced version $\partial \CV^r$ is homotopy equivalent to the \emph{boundary of Jewel space}. By \cite{bsv}, this reduced boundary  $\partial \CV^r$ is also the boundary of the bordification of reduced Outer space used in Bestvina and Feighn's proof that $\out(F_g)$ is a virtual duality group. Beyond $\partial \CV$ and $\partial \CV^r$, a third analogue of the Tits building in this context is the free factor complex, which in general is known not to present the dualising module \cite{hmnp}. 

We remark that the distinction between reduced and unreduced Outer space is often blurred in the literature. When considering their simplicial boundary, it is important to distinguish between them: $\partial \CV \not\simeq \partial\CV^r $.

\end{rem}

\subsection{An outline of the proofs}

To ease the outline, assume $V$ satisfies $g\geq 2$ and $b=p=0$. A group $G$ is a virtual duality group if it satisfies the following three conditions (Corollary~\ref{criterion} below): 
\begin{enumerate}
\item $G$ is virtually torsion-free.
\item $G$ acts properly and cocompactly on a contractible
orientable $d$-dimensional topological manifold-with-boundary $M$.
\item $\partial M$ is homology equivalent to a wedge of spheres of constant dimension.
\end{enumerate}
If (1)---(3) hold, then $\widetilde H_\ast(\partial M,\Z)$ is the dualising module of $G$. We apply this to $G=\hmod(V)$. The fact that $\hmod(V)$ is virtually torsion-free follows from the fact that it is a subgroup of $\mcg(S)$, which is virtually torsion-free \cite[Theorem~6.9]{farbmargalit}. 

Recall that a \emph{meridian} on $S=\partial V$ is a curve bounding a disc in $V$. Discs are determined up to isotopy by their meridians. We pick constants $\eps < \eps ' < \log(3+\sqrt{8})$ and take $M$ to be the subspace of the Teichm\"uller space of $S$ consisting of marked surfaces $\sigma$ where each simple closed curve on $\sigma$ has length at least $\eps$, and the set of curves of length at most $\eps'$ on $\sigma$ form a \emph{simple system of meridians} (the requirement $\eps ' < \log(3+\sqrt{8})$ ensures that the set of curves of length at most $\eps'$ are disjoint). The space $M$ is a truncated cover of the orbifold classifying space of $\hmod(V)$ constructed by Hainaut and the first author in \cite{hainautpetersen}. The work in \cite{hainautpetersen} implies that $M$ is a contractible manifold with boundary (Theorem~\ref{hp}), and the truncation ensures the action is cocompact. 

We can write $\partial M = N_1 \cup N_2$, where $N_1$ parametrises hyperbolic surfaces with a closed geodesic of length $\eps$, and $N_2$ parametrises hyperbolic surfaces with the property that the closed geodesics of length \emph{strictly less} than $\eps'$ form a \emph{non-simple} system of meridians. We show that:
\begin{enumerate}[(i)]
\item $N_1$ is contractible,
\item $N_2$ is contractible,
\item $N_1 \cap N_2$ is homotopy equivalent to $\compNS$. 
\end{enumerate}
Consequently, $\partial M$ is homotopy equivalent to the suspension $\Sigma \compNS$. The results (i)---(iii) are established by considerations similar to those of Borel--Serre \cite{borelserre} and Harer \cite{harer}: we cover $\partial M$ by closed contractible pieces, and deduce that $\partial M$ is homotopy equivalent to the nerve of the covering. 
Concretely, we obtain a poset --- we denote it $\dRGB$ --- such that $|\dRGB| \simeq \partial M$, and such that $N_1$, $N_2$, and $N_1 \cap N_2$ are homotopic to realizations of subposets of $\dRGB$. These subposets are analyzed using poset arguments in the spirit of Quillen \cite{quillen-homotopyproperties} to show (i)---(iii). 

After the above, Theorems~\ref{thm:A}, \ref{thm:B}, and ~\ref{thm:C} all follow after showing that $\compNS$ is homology equivalent to a wedge sum of copies of $S^{2g-3}$. This is achieved via two steps:
\begin{itemize}
    \item $\compNS$ is homologically $(2g-4)$-connected.
    \item $\compNS$ is homology equivalent to a $(2g-3)$-dimensional CW complex.
\end{itemize}The first bullet point is easy: we already know that the virtual cohomological dimension of $\mcg(V)$ is $4g-5$, so the compactly supported cohomology $H^\ast_c(M,\Z)$ of $M$ vanishes for $\ast > 4g-5$. As $M$ is $(6g-6)$-dimensional, Lefschetz duality implies that $\partial M \simeq \Sigma \compNS$ is $(2g-3)$-acyclic. 

The second bullet point is carried out by an adaptation of the strategy used  in the work of Ivanov \cite{ivanov} to prove that the curve complex is homotopy equivalent to a CW complex of dimension at most $2g-2$.

\subsection{Acknowledgements}
We thank Benjamin Br\"uck for bringing this problem to our attention. We thank Sebastian Hensel and Andrew Putman for many helpful discussions and their involvement in the early stages of this project. We thank Karen Vogtmann for kindly sharing an early version of \cite{vogtmann}. 

Petersen is supported by a Wallenberg Scholar fellowship. Wade is supported by the Royal Society of Great Britain through a University Research Fellowship.

\section{Background on (virtual) duality groups}

In this section we recall some well-known facts about duality groups, for future reference and for the reader's convenience. The notion of a duality group was introduced by Bieri and Eckmann \cite{be}, and nearly all that follows can be found in their original article. 

\begin{defn}
    Let $G$ be a group. We say that $G$ is a \emph{duality group} of \emph{cohomological dimension} $n$, with \emph{dualising module} $C$, if there exists a class $e \in H_n(G,C)$ such that cap-product with $e$ defines a natural isomorphism of functors 
    \[ H^k(G,-) \stackrel \cong \longrightarrow H_{n-k}(G,C \otimes_\Z - )\]
    for all integers $k$. 
\end{defn}

\begin{rem}If $C=\mathbb{Z}$ (as an abelian group, possibly with nontrivial $G$-action), then one arrives at the special case of a \emph{Poincar\'e duality group}. A source of geometric examples of Poincar\'e duality groups of dimension $n$ are the fundamental groups of aspherical closed manifolds of dimension $n$, in which case $C$ is the rank one local system given by the orientation sheaf. Allowing $C$ to have larger rank leads to a significantly richer theory. \end{rem}

Duality for a group of type $\mathit{FP}$ can be detected by cohomology with coefficients in the group ring.

\begin{thm}\label{b-e criterion}
    The following two properties are equivalent for a group $G$.
    \begin{enumerate}
        \item $G$ is a duality group of dimension $n$.
        \item $G$ is of type $\mathit{FP}$, $H^*(G,\Z[G])$ is torsion-free, and $H^k(G,\Z[G])=0$ unless $k=n$.
    \end{enumerate}
When this holds, $H^n(G,\Z[G])$ is the dualising module of $G$. 
\end{thm}

\begin{proof}
    All parts of the implication (1) $\Rightarrow$ (2) is proven in \cite[Section 1]{be}, except for $G$ necessarily being of type $\mathit{FP}$, which is due to Brown \cite{brown-finiteness}. The implication (2) $\Rightarrow$ (1) is \cite[Theorem 4.5]{be}.
\end{proof}

\cref{b-e criterion} implies in particular that the dualising module and the dimension are uniquely determined by the group $G$.

Duality groups have finite homological dimension and are in particular torsion-free. The theory is however still useful for groups which are merely virtually torsion-free; namely, one is led to the following definition:

\begin{defn}
    A group $G$ is a \emph{virtual duality group} if it admits a finite index subgroup which is a duality group. 
\end{defn}

\begin{prop}\label{finite index}
    Let $G$ be a torsion-free group. Then $G$ is a duality group if and only if any one of its finite-index subgroups is a duality group. The dualising module of each finite-index subgroup is then canonically isomorphic. 
\end{prop}

\begin{proof}
    \cite[Theorem 3.2]{be} and \cite[Theorem 3.3]{be}.
\end{proof}

\cref{finite index} thus says that if \emph{some} finite index subgroup of a group $G$ is a duality group, then \emph{all} torsion-free finite index subgroups of $G$ are duality groups, and they all have the same cohomological dimension and dualising module. In particular, the dimension and dualising module of a virtual duality group are well-defined.

\begin{prop}\label{prop:SES}
    Let $1 \to K \to G \to H \to 1$ be a short exact sequence of groups. If $K$ is a duality group of dimension $k$ and $H$ is a (virtual) duality group of dimension $l$, then $G$ is a (virtual) duality group of dimension $k+l$, and the dualising module for $G$ is the tensor product of the dualising modules for $K$ and $H$.
\end{prop}

\begin{proof}For actual duality groups this is \cite[Theorem 3.5]{be}. The virtual case follows by choosing a finite-index duality subgroup $H' \leq H$, taking $G'$ to be the preimage of $H'$ in $G$, and considering $1 \to K \to G' \to H' \to 1$. 
\end{proof}

\begin{prop}\label{action}
    Let $G$ be a group acting freely, properly and cocompactly on a contractible orientable $d$-dimensional topological manifold-with-boundary $M$. Then there are isomorphisms
    $$ H^k(G,\Z[G]) \cong \widetilde H_{d-1-k}(\partial M,\Z)$$
    for all $k$. 
\end{prop}

\begin{proof}
    \cite[\S 6.4]{be}.
\end{proof}

\begin{cor}\label{criterion}
    Let $G$ be a (virtually) torsion-free group acting properly and cocompactly on a contractible orientable $d$-dimensional topological manifold-with-boundary $M$. Then $G$ is a (virtual) duality group if and only if $\partial M$ is homology equivalent to a wedge of spheres of equal dimension $q$. When this holds, the dualising module is $\widetilde H_q(\partial M,\Z)$, and the (virtual) cohomological dimension is $d-q-1$.
\end{cor}

\begin{proof}This is \cite[Theorem 6.2]{be}. If $G$ is torsion-free then $G$ acts freely on $M$, and $G$ is of type $\mathit{FP}$ since it admits a classifying space $M/G$ of the homotopy type of a finite CW complex. The result is then immediate from \cref{b-e criterion} and \cref{action}. If $G$ is only virtually torsion-free then we merely replace $G$ by a finite-index torsion-free subgroup.
\end{proof}

\cref{criterion} generalises the fact that fundamental groups of closed aspherical manifolds are Poincar\'e duality groups with trivial dualising module (noting that the reduced homology of the empty space is $\Z$ in degree $-1$). The geometric criterion of \cref{criterion} leads to large classes of examples of (virtual) duality groups. In particular:
\begin{enumerate}
    \item Arithmetic groups are virtual duality groups. One chooses for $M$ the Borel--Serre bordification \cite{borelserre} of the associated symmetric space. The fact that the boundary has the homotopy type of a wedge of spheres is the Solomon--Tits theorem.
    \item The mapping class group of an orientable surface is a virtual duality group \cite{harer,ivanov}. One chooses for $M$ the \emph{thick part of Teichm\"uller space} (which will be discussed further shortly); one could alternatively use Harvey's bordification of Teichm\"uller space. The fact that the boundary has the homotopy type of a wedge of spheres is the heart of Harer's proof.
    \item Fundamental groups of knot complements are duality groups (see Example 1 in Section VIII.9 of \cite{brown2}: this also contains some helpful, more general, discussion). One chooses for $M$ the universal cover of the complement of a tubular neighborhood of a knot in $S^3$. The fact that knot groups are duality groups can also be seen more directly. 
\end{enumerate}

Our proof that the handlebody mapping class group is a virtual duality group will likewise be an application of the above criterion. 

\section{A combinatorial model of a classifying space for the handlebody group}

\subsection{Teichm\"uller space and thick Teichm\"uller space}

This material is covered in Chapter~10 of \cite{farbmargalit}. To expand on the notation given in the introduction, throughout $V$ will denote  a handlebody of genus $g$, with $b$ marked discs on its boundary, and $p$ marked points on its boundary, all of which are disjoint. We let $S$ denote the complement of the interiors of the marked discs and points in $\partial V$. Hence $S$ is an oriented surface of genus $g$ with $b$ boundary components, and $p$ punctures. A simple closed curve on $S$ is \emph{essential} if it does not bound a disc in $S$, and \emph{nonperipheral} if it does not bound a punctured disc or annulus in $S$.

We let $\mcg(S)$ denote the mapping class group of self-diffeomorphisms of $S$ fixing the punctures, and a collar neighborhood of the boundary, pointwise. We let $\hmod(V)$ denote the \emph{handlebody group}, i.e.\ the mapping class group of self-diffeomorphisms of $V$ fixing the marked boundary discs and marked points.  

Let $\T(S)$ be the \emph{Teichm\"uller space} of the surface $S$. This is the space of marked, complete, hyperbolic metrics on $S$, with totally geodesic boundary. It is diffeomorphic to an open ball of dimension $6g-6+3b+2p$. If $b=0$, then the group $\mcg(S)$ acts properly on $\T(S)$, and the quotient of $\T(S)$ by $\mcg(S)$ is Riemann's moduli space $\mathcal M_{g,p}$ of compact Riemann surfaces of genus $g$ with $p$ marked points. 

If $\mathbf b \in \mathbb R^b_{>0}$ is a vector, then we denote by $\T(S;\mathbf b)$ the subspace of $\T(S)$ where the hyperbolic lengths of the boundary components are exactly given by $\mathbf b$. 

The group $\mcg(S)$ does not act properly on $\T(S)$ when $b>0$, since a Dehn twist around a boundary component acts trivially. To remedy this, we let $\TT(S)$ denote $\T(S) \times \R^b$, where we have informally speaking introduced an additional Fenchel--Nielsen twist parameter along each boundary component. More intrinsically, if $\T(S)$ is the universal cover of the moduli space of hyperbolic surfaces of type $(g,b,p)$, then $\TT(S)$ is the universal cover of the moduli space which parametrises in addition the datum of a marked point on each boundary component. We define $\TT(S;\mathbf b)$ similarly.

Fix a positive real number $\eps$ with $\eps<\log(3+\sqrt 8)$. We write $\TT_\eps(S;\mathbf b)$ for the \emph{$\eps$-thick Teichm\"uller space}, by which we mean the locus inside $\TT(S;\mathbf b)$ parametrising hyperbolic metrics with the property that all essential, nonperipheral, closed geodesics on $S$ have length at least $\eps$. The space $\TT(S;\mathbf b)$ admits a $\mcg(S)$-equivariant deformation retraction down to $\TT_\eps(S;\mathbf b)$ \cite{jiwolpert}. The space $\TT_\eps(S;\mathbf b)$ is naturally a smooth manifold-with-corners of dimension $6g-6+3b+2p$, and $\mcg(S)$ acts on it properly and cocompactly.

\subsection{Handlebody Teichm\"uller space}
We now describe a classifying space for the handlebody group in terms of Teichm\"uller theory. The following construction was introduced by Hainaut and the first named author \cite{hainautpetersen}. For background on handlebody mapping class groups, see \cite{hensel}.

Fix a second real number $\eps'$ with $\eps<\eps'<\log(3+\sqrt 8)$, and an arbitrary length vector $\mathbf b \in \mathbb R^b_{>0}$. Let $\Y(V) \subset \TT_\eps(S;\mathbf b)$ denote the closed subspace of the $\eps$-thick Teichm\"uller space parametrising hyperbolic metrics on $S$ with the following property:
\emph{all essential, nonperipheral, closed geodesics on $S$ of length at most $\eps'$ bound discs in $V$, and the collection of these discs form a simple disc system in $V$. }

\begin{thm}[Hainaut--Petersen]\label{hp}
    The space $\Y(V)$ is a contractible smooth oriented manifold-with-corners of dimension $6g-6+3b+2p$, and the group $\hmod(V)$ acts properly and cocompactly on $\Y(V).$
\end{thm}

 \cref{hp} is not stated explicitly in \cite{hainautpetersen}, so let us comment on this. The fact that $\Y(V)$ is a manifold-with-corners is an easy local verification in Fenchel--Nielsen coordinates (taking coordinates defined by a pants decomposition containing all short geodesics), once one knows the following \cref{lemma:wolpert}. We refer to \cite[Section 2]{hainautpetersen} for more details.

\begin{lem}\label{lemma:wolpert}
    There is an open cover of $\,\T(S)$ such that on each open set, there are only a finite number of simple closed curves on $S$ for which the associated length function takes values $\leq \eps'$, and they are disjoint.  
\end{lem}

\begin{proof}This follows from the Collar Lemma and Wolpert's Lemma. See the discussion following \cite[Theorem 2.3]{hainautpetersen}.
\end{proof}
What remains of \cref{hp} is contractibility of $\Y(V)$. Let $\mathcal{H}_\eps\TT(V)$ be the open subspace of $\TT(S)$ defined by the condition that all closed geodesics on $S$ of length strictly less than $\eps$ bound discs in $V$, and the collection of these discs form a simple disc system in $V$. Let $\overline{\mathcal{H}}_\eps\TT(V)$ be its closure, i.e.\ the locus defined by the condition that closed geodesics on $S$ of length at most $\eps$ form a simple disc system in $V$. The space $\mathcal{H}_\eps\TT(V)$ is contractible according to \cite[Theorem 1.4]{hainautpetersen}, and \cite[Proposition 3.5]{hainautpetersen} says that both maps
 \[ \mathcal{H}_\eps\TT(V) \hookrightarrow \overline{\mathcal{H}}_\eps\TT(V) \hookleftarrow (\overline{\mathcal{H}}_\eps\TT(V)  \cap \TT_\eps(S;\mathbf b))\]
 are homotopy equivalences. But now we may observe that $\Y(V) = \overline{\mathcal{H}}_{\eps'}\TT(V)  \cap \TT_\eps(S;\mathbf b)$, and none of the arguments need to be changed if the two parameters $\eps$ and $\eps'$ are chosen to be distinct. 

 A final remark is that the discussion in \cite{hainautpetersen} is consistently formulated in terms of the quotient orbifold $\mathcal{H}_\eps\TT(V)/\hmod(V)$, rather than its universal cover. An alternative proof of the contractibility of $\Y(V)$ follows from the $\RGB^\textnormal{op}$-stratification of $\Y(V)$ described below: see Remark~\ref{rem:alternative_contractibility_proof}.

\subsection{The RGB poset}

In this section, we shall construct a stratification of the space $\mathcal Y(V)$, whose strata are indexed by coloured simple disc systems in $V$. 

\begin{defn}[The posets $\RGB$ and $\dRGB$] Define the poset $\RGB=\RGB(V)$ as follows. An element of $\RGB(V)$ is a given by a simple disc system in $V$, all of whose discs are coloured red, green or blue. If $g(V)=0$, then the empty disc system is considered to be a simple disc system. We say $D \leq D'$, if $D'$ can be obtained from $D$ by a combination of adding new blue discs, turning some green discs blue, and turning some green discs red. We denote by $\dRGB(V) \subset \RGB(V)$ the upwards-closed subposet consisting of disc systems which either contain a red disc, or with the property that deleting all blue discs leaves a non-simple disc system. 
\end{defn}

\begin{defn}[The \emph{RGB disc system} given by $\sigma \in \mathcal{Y}(V)$] \label{def of RGB stratification}
For $\sigma \in \mathcal{Y}(V)$, its associated \emph{RGB disc system} $d(\sigma)$ has an underlying simple disc system consisting of discs whose associated geodesic meridians have length at most $\eps'$. Furthermore
\begin{itemize}
	\item discs associated to geodesics of length $\eps$ are red,
	\item discs associated to geodesics of lengths in $(\eps,\eps')$ are green,
	\item discs associated to geodesics of length $\eps'$ are blue. 
\end{itemize}
\end{defn}

\begin{defn}
Let $\Po$ be a poset. The \emph{Alexandrov topology} on  $\Po$ is the topology where a set is open if and only if it is upwards-closed in $\Po$. A \emph{$\Po$-stratification} of a topological space $X$ is a continuous map $f \colon X \to \Po$ with respect to the Alexandrov topology on $\Po$. Equivalently, $f \colon X \to \Po$ is a $\Po$-stratification if and only if $f^{-1}(\Po_{\geq p})$ is open for all $p \in \Po$. 
\end{defn}

\begin{defn}Let $\Po$ be a poset. We use $\Po^\op$ to denote the \emph{opposite poset}, where the ordering is reversed. We will often care about posets only via their geometric realization; we note that the geometric realizations of $\Po$ and $\Po^\op$ are isomorphic as simplicial complexes.
\end{defn}

\begin{thm}\label{thm:combinatorial model}
    The map $d \colon \Y(V) \to \RGB(V)^\textnormal{op}$ of \cref{def of RGB stratification} is a $\hmod(V)$-equivariant $\RGB^\textnormal{op}$-stratification of $\,\Y(V)$. The boundary $\partial \Y(V)$ is the union of strata given by $d^{-1}(\dRGB^\textnormal{op})$. 
\end{thm}

\begin{proof}Continuity of $d$ follows from the fact that the length function associated to every curve on $S$ is a continuous function $\T_\eps(S)\to \R_{\geq \eps}$. At the risk of being overly formal, we may consider the stratification of $\R_{\geq \eps}$ with two closed strata $\{\eps\}$ and $\{\eps'\}$, and two open strata $(\eps,\eps')$ and $(\eps',\infty)$. Then the RGB-stratification of $\Y(V)$ is the pull-back of the product of these stratifications along the map $\Y(V) \to \R_{\geq \eps}^\infty$ given by the product of all length functions. The stratification is $\hmod(V)$-equivariant, since the map $\Y(V) \to \R_{\geq \eps}^\infty$ is equivariant.

Now take $\sigma \in \Y(V)$ with $d(\sigma)=D\in \dRGB^\op$. If $D$ contains a red curve, then $\sigma$ lies in the boundary of $\eps$-thick Teichm\"uller space, and in particular also in the boundary of $\Y(V)$. If deleting the blue discs from $D$ leaves a non-simple disc system, then one can find an arbitrarily small perturbation $\sigma'$ of $\sigma$ within Teichm\"uller space such that all curves on $S$ associated with the blue discs have length function strictly larger than $\eps'$ on $\sigma'$, and in particular it follows that $\sigma'$ lies on the boundary of $\Y(V)$. Conversely, one verifies that if $d(\sigma) \notin \dRGB^\op$ then an open neighbourhood of $\sigma$ in Teichm\"uller space lies in $\Y(V)$, and $\sigma$ must be in the interior of $\Y(V)$.
\end{proof}

\begin{defn}\label{ht of element}
    The \emph{height} $\Ht(p)$ of an element $p$ in a poset $\Po$ is the size of the longest chain in $\Po_{<p}$. 
\end{defn} 
\begin{prop}\label{prop:strat}Let $\Po$ be a poset such that every element $p \in \Po$ has finite height. Suppose $f \colon X \to \Po$ is a $\Po$-stratification of a topological space $X$. Assume the following conditions:
\begin{enumerate}
    \item The stratification of $X$ is locally finite, in the sense that $X$ admits a basis of open sets each of which meets only finitely many closures of strata.
    \item $f^{-1}(\Po_{\leq p})$ is contractible for all $p \in \Po$.
    \item Each inclusion
    $$ f^{-1}(\Po_{<p}) \hookrightarrow f^{-1}(\Po_{\leq p})$$
    is a cofibration. 
  
\end{enumerate}Then $X$ is weakly homotopy equivalent to the geometric realization $|\Po|$.
\end{prop}

\begin{proof}
Consider the diagram (i.e.~functor) $\Po \to \mathbf{Top}$ given by $p \mapsto f^{-1}(\Po_{\leq p})$. Condition (1) implies that the colimit of this diagram is $X$. Indeed, this means precisely that $X$ carries the colimit topology from the union of subspaces $f^{-1}(\Po_{\leq p})$, which is the case for any locally finite closed covering of a topological space. Condition (2) tells us that every subspace $f^{-1}(P_{\leq p})$ is contractible, so that the homotopy colimit of this diagram is $|\Po|$. Therefore it is enough to show that the colimit and homotopy colimit are weakly homotopy equivalent in this case.

Since every element of $\Po$ has finite height, the height function $\Ht \colon \Po \to \mathbb{N}$ endows $\Po$ with the structure of a \emph{direct category} in the sense of \cite[\S5.1]{hovey}. This implies that the projective model structure on $\mathrm{Fun}(\Po,\mathbf{Top})$ exists, and has the property that $\Phi :\Po \to \mathbf{Top}$ is cofibrant if and only if, for each $p \in \Po$, the canonical map
\[ \operatornamewithlimits{colim}_{q<p} \Phi(q) \to \Phi(p)\]
is a cofibration \cite[Theorem 5.1.3]{hovey}. In particular, our diagram is cofibrant, since this condition precisely agrees with (3). It follows that its colimit computes the homotopy colimit, since the colimit is a left Quillen functor for this model structure.\end{proof}

\begin{thm}\label{thm:combinatorial model 2} The stratification of $\,\Y(V)$  by $\RGB(V)^\op$ described in \cref{thm:combinatorial model} satisfies the conditions of \cref{prop:strat}. In particular, $|\RGB(V)| \simeq \Y(V)$. Similarly, the stratification of $\,\partial \Y(V)$ by $\partial\RGB(V)^\op$ satisfies the same conditions, and $|\partial \RGB(V)| \simeq \partial\Y(V)$.

    \end{thm}

\begin{proof}The proof is the same in both cases.
Local finiteness of the stratification is clear from \cref{lemma:wolpert}. Let us describe the closures of strata. Take $D \in \RGB(V)$, and suppose that $D$ cuts $V$ into genus zero handlebodies $V_1,\dots,V_s$, each with some marked discs and points on its boundary. If $D$ consists of $k$ discs, then the group $\R^k$ acts on the cartesian product 
 $\prod_{i=1}^s \TT(\partial V_i)$ by shifting the twist parameters along each disc in $D$. More precisely, each disc in $D$ corresponds to two distinct marked boundary discs among the handlebodies $V_1,\dots,V_s$, and $x \in \R$ acts by shifting the twist parameter on one of these discs by a factor $x$, and the other disc by a factor $-x$. The closure of the stratum corresponding to $D$ is then naturally identified with the subspace of 
$$\prod_{i=1}^s \TT_{\eps'}(\partial V_i)/\R^k$$
defined by imposing the evident conditions on the lengths of boundary components:
\begin{enumerate}
    \item Boundary components corresponding to red (resp.\ blue) discs have length $\eps$ (resp.~$\eps'$).
    \item Boundary components of the original handlebody $V$ have lengths as prescribed by the fixed length vector $\mathbf b$ appearing in the definition of $\Y(V)$.
    \item Boundary components corresponding to green discs have length in the interval $[\eps,\eps']$, and the two boundary components corresponding to each green disc have the same length.
\end{enumerate}
From this description it is not hard to see (using Fenchel--Nielsen coordinates) that the closure of the stratum is a smooth manifold-with-corners, that the stratum itself is precisely the interior of the manifold, and that the interior is diffeomorphic to a ball. In particular the closure is contractible, and the inclusion of the boundary is a cofibration.
\end{proof}

\section{The RGB poset and the non-simple disc complex} \label{s:poset}

Let $\NS(V)$ denote the poset of nonempty, non-simple disc systems in $V$. The realization $|\NS(V)|$ is naturally identified with the barycentric subdivision of the simplicial complex $\compNS(V)$ defined in the introduction. The goal of this section is to prove the following. 
\begin{thm}\label{thm:simplifying}Let $V$ be a handlebody of genus $g>0$. Then 
    $|\dRGB(V)| \simeq \Sigma |\NS(V)| \cong \Sigma \compNS(V)$. 
\end{thm}

 It will be convenient to freely use topological terminology regarding posets in this section. In particular, when we say e.g.\ that a poset is contractible, or that a map of posets is a homotopy equivalence, we mean that this holds for the geometric realizations.

\subsection{Poset lemmas}
Let us for the reader's convenience recall the tools of combinatorial topology of posets which we shall use. All results below are standard.

\begin{lem}[Quillen's fibre lemma] \label{lemma:fibre_lemma} Let $f:\mathsf P \to \mathsf Q$ be a map of posets. Suppose that either $f^{-1}(\mathsf Q_{\geq x})$ is contractible for all $x \in \mathsf Q$, or $f^{-1}(\mathsf Q_{\leq x})$ is contractible for all $x \in \mathsf Q$. Then $f \colon \mathsf P\to\mathsf Q $ is a homotopy equivalence. \end{lem}

\begin{proof}
    See \cite[Proposition 1.6]{quillen-homotopyproperties}.
\end{proof}

The following notion does not appear to have a standard name, even though \cref{monotone} is well-known. 

\begin{defn}
    Let $\Po$ be a poset. A map of posets $r:\Po \to \Po$ is said to be \emph{inflating} if $x \leq r(x)$ for all $x \in \Po$, and \emph{deflating} if $x \geq r(x)$ for all $x \in \Po.$
\end{defn}

\begin{lem}\label{monotone}Let $\mathsf P$ be a poset, and let $r\colon\mathsf P\to \mathsf P$ be inflating or deflating. Then $r\colon\mathsf P\to \mathrm{Im}(r)$ is a homotopy equivalence.
\end{lem}

\begin{proof}Let $i:\mathrm{Im}(r) \to \Po$ be the inclusion. We need to show that $i \circ r$ and $r \circ i$ are both homotopic to the identity. 
    Now use that if $f(x) \leq g(x)$ for all $x$, then $f$ and $g$ are homotopic \cite[\S1.3]{quillen-homotopyproperties}.
\end{proof}

\begin{notation}\label{notation:quillenjoin}
    If $\mathsf P$ and $\mathsf Q$ are posets, then $\mathsf P \ast \mathsf Q$ denotes the poset with underlying set $\mathsf P \sqcup\mathsf Q$, where the partial orders on the subsets $\mathsf P$ and $\mathsf Q$ are the same as in the original posets, and $p < q$ for all $p \in \mathsf P$ and $q \in \mathsf Q$. There is a homeomorphism  $|\mathsf P \ast\mathsf Q| \cong |\mathsf P| \ast |\mathsf Q|$, justifying the notation, and the operation $\ast$ is often called the \emph{join} of posets, following Quillen \cite[p.~104]{quillen-homotopyproperties}.
\end{notation}

\begin{notation}Let ${\partial\mathsf\Delta}^k$ denote the poset of nonempty proper subsets of $\{0,1,\dots,k\}$. The notation is justified by the fact that $|\partial\mathsf\Delta^k| \cong S^{k-1}$ is the boundary of a $k$-simplex. 
\end{notation}

\begin{lem}\label{hocolim2}
    Let $\mathsf P$ be a poset. Suppose we can write $\mathsf P = \mathsf Q_0 \cup  \mathsf Q_1 \cup \cdots \cup \mathsf Q_k$ with each $\mathsf Q_i$ downwards-closed. If  $\bigcap_{i \in I}\mathsf Q_i$ is contractible for every proper, nonempty subset $I \subset \{0, 1,\ldots, k\}$, then$$ |\mathsf P| \simeq \Sigma^{k} |\mathsf Q_0 \cap \mathsf Q_1\cap \cdots \cap \mathsf Q_k|.$$
\end{lem}

\begin{proof}
    Let $\mathsf{R}=(\mathsf Q_0 \cap \mathsf Q_1\cap \cdots \cap \mathsf Q_k) \ast \partial \mathsf \Delta^{k}$. We define a map 
    $ f\colon \mathsf P \to \mathsf{R}$ by 
\[ f(x)= \begin{cases} x & \text{ if $x\in \mathsf Q_0 \cap \mathsf Q_1\cap \cdots \cap \mathsf Q_k$,}\\
\{ i : x \not\in \mathsf Q_i \} & \text{ otherwise. } \end{cases} \]
As each $\mathsf Q_i$ is downwards closed it follows that $f$ is a poset map. If $r \in \mathsf R$ belongs to $\mathsf Q_0\cap \mathsf Q_1 \cap \cdots \cap \mathsf Q_k$ then  $f^{-1}(\mathsf R_{\leq r})=\Po_{\leq r}$, which is contractible as it has $r$ as a maximal element. Otherwise $f^{-1}(\mathsf R_{\leq r})=\bigcap_{i \not\in r} \mathsf Q_i$ and $f^{-1}(\mathsf R_{\leq r})$ is contractible by the hypothesis.  Hence $f$ is a homotopy equivalence by the  fibre lemma (\ref{lemma:fibre_lemma}). It follows that 
    \[|\mathsf P| \simeq |\mathsf Q_0 \cap \mathsf Q_1\cap \cdots \cap \mathsf Q_k| \ast S^{k-1} \cong \Sigma^{k}|\mathsf Q_0 \cap \mathsf Q_1\cap \cdots \cap \mathsf Q_k|. \qedhere\]
\end{proof}

\begin{cor}\label{hocolim1}If we have $\mathsf P = \mathsf Q_0 \cup  \mathsf Q_1$ with each $\mathsf Q_i$ downwards-closed, and $\mathsf Q_0$ and $\mathsf Q_1$ are both contractible, then $|\mathsf P| \simeq \Sigma |\mathsf Q_0 \cap \mathsf Q_1|.$
\end{cor}

\begin{rem}
    The preceding corollary could also be proven by observing that the evident commutative square
 \[\begin{tikzcd}
     \vert\mathsf Q_0 \cap \mathsf Q_1\vert \arrow[r]\arrow[d]& \vert\mathsf Q_0 \vert\arrow[d] \\
     \vert\mathsf Q_1\vert \arrow[r]& \vert\mathsf P\vert
 \end{tikzcd}\]
 (which is a pushout square on the level of spaces) is a homotopy pushout. Similarly \cref{hocolim2} could be proven by using that $|\mathsf P|$ is the homotopy colimit of the diagram of spaces given by the family $|\bigcap_{i \in I}\mathsf Q_i |$, and applying \cite[Example 5.7.5]{munson-volic}.\end{rem}

\subsection{Simplifying the RGB poset}

Let $\D(V)$ denote the poset of nonempty disc systems in the handlebody $V$, ordered by inclusion. Thus $|\D(V)|$ is the barycentric subdivision of the simplicial complex $ \mathbf D(V)$ considered in the introduction. 

The following theorem is proven in \cite[Theorem 5.3]{mccullough} when there are no marked points or discs on the boundary of the handlebody. See also \cite{cho}. The general case is proven in \cite[Proposition 8.1]{hatcherwahl} by an induction on the number of marked points/discs, thereby reducing to the case considered by McCullough.

\begin{thm}[McCullough, Hatcher--Wahl]\label{disk cx contractible}
    Suppose that $g>0$. Then $\D(V)$ is contractible.
\end{thm}

\begin{defn}Let 
$\SD(V)$ denote the poset of \emph{simple} disc systems in the handlebody $V$, ordered by inclusion. If the genus of $V$ is zero, then the empty disc system is simple, and is a minimal element of $\SD(V)$.
\end{defn}

Hence, if $g>0$, then $\mathsf D(V)$ is the union of the disjoint subposets $\SD(V)$ and $\NS(V)$, the former being upwards-closed and the latter being downwards-closed. When $g=0$, $\mathsf S(V)=\mathsf D(V) \cup \{\varnothing\}$, and $\NS(V)=\varnothing$. 

It follows from \cref{disk cx contractible} that $\SD(V)$ is always contractible, an observation due to Giansiracusa \cite[Proposition 6.3]{giansiracusa}. We recall his argument.

\begin{prop}\label{prop: S contractible}
    The realization $|\SD(V)|$ is contractible.
\end{prop}

\begin{proof}
    As just noted, $\SD(V)$ has a minimal element if $g=0$. For $g>0$ it suffices by \cref{disk cx contractible} to show that the natural inclusion $f:\SD(V)\to\D(V)$ is a homotopy equivalence. We apply Quillen's fibre lemma (\ref{lemma:fibre_lemma}). If $D \in \D(V)$ cuts $V$ into smaller handlebodies $V_1,\dots,V_s$ (each with some number of marked points and discs on its boundary), then 
    \[ f^{-1}(\D(V)_{\geq D}) \cong \prod_{i=1}^s \SD(V_i).\](The cartesian product of posets is ordered coordinatewise.)
    But each factor $\SD(V_i)$ may be assumed contractible by induction on the quantity $-\chi(S) = 2g-2+b+p$. 
\end{proof}

For $D \in \RGB$, let $R(D)$ and $B(D)$ be the subsystems of red and blue discs, respectively.

\begin{prop}\label{p:colour_forgetting_maps}
The poset map $f: \RGB \to \SD(V)$ given by forgetting colours is a homotopy equivalence. If $g>0$, then the same is true for the restrictions of $f$ to the subposets
\begin{align*}
{\mathsf Q_0} &= \{D \in \RGB : f(D) \setminus B(D) \in \NS(V) \} \text{ and} \\
{\mathsf Q_1} &= \{D \in \RGB : R(D) \neq \emptyset \},
\end{align*}
respectively. Hence $\RGB$ is contractible, and so are $\mathsf Q_0$ and $\mathsf Q_1$ when $g>0$.
\end{prop}

\begin{proof}For $D \in \SD(V)$, let $D_{\mathrm{blue}}$ and $D_{\mathrm{red}}$ denote the elements of $\RGB$ given by the disc system $D$ with all discs coloured blue or red, respectively.

For the first part, it suffices by the fibre lemma to show that $f^{-1}(\SD(V)_{\geq D})$ deformation retracts to $\{D_{\mathrm{blue}}\}$ for all $D \in \SD(V)$. We do this in three steps, using \cref{monotone}. For $E \in f^{-1}(\SD(V)_{\geq D})$, first turn all red discs in $E$ green (which is deflating), then turn all green discs blue (which is inflating), and finally forget all blue discs that do not belong to $D$ (which is deflating). Each of these maps is a homotopy equivalence onto its image by \cref{monotone}, so that $f^{-1}(\SD(V)_{\geq D})$ is contractible. 

Now assume that $g>0$. If $E \in \RGB$, then turning red discs green or deleting blue discs does not change $f(E) \setminus B(E)$, and turning green discs blue makes $f(E) \setminus B(E)$ smaller. As a result, the three operations above preserve the property of an element of $\RGB$ being in $\mathsf Q_0$.  Hence the same argument shows also that $\mathsf Q_0 \cap f^{-1}(\SD(V)_{\geq D}) \simeq \{D_{\mathrm{blue}}\}$, so that also $\mathsf Q_0 \to \SD(V)$ is a homotopy equivalence.

For $\mathsf Q_1$ one needs to take a slightly different approach and build a deformation retract of $\mathsf Q_1 \cap f^{-1}(\SD(V)_{\geq D})$ onto $\{D_{\mathrm{red}}\}$. For $E \in \mathsf Q_1 \cap f^{-1}(\SD(V)_{\geq D})$ we first turn all blue discs of $E$ that belong to  $ D$ green (deflating), and then turn all green discs of $E$ that belong to $ D$ red (inflating). At this point, all discs of $E$ that belong to $D$ are red. Then all red discs in $E$ that do not belong to $D$ are turned green (deflating), green discs in $E$ that do not belong to $D$ are turned blue (inflating), and blue discs in $E$ that do not belong to $D$ are deleted (deflating). \end{proof}

\begin{rem}
    Let us make explicit where the assumption $g>0$ is used in the preceding proof. The argument in the third paragraph fails when $g=0$, since in that case $D_{\mathrm{blue}} \notin \mathsf Q_0$ and in fact $\mathsf Q_0 = \varnothing$. The argument in the fourth paragraph breaks down when $g=0$, since in this case we may have $D=\varnothing$ in $\SD(V)$, in which case $D_{\mathrm{red}} \notin \mathsf Q_1$.
\end{rem}

\begin{rem}
    It may be helpful to think heuristically about the deformation retraction onto $\{D_{\mathrm{blue}}\}$ in the preceding proposition as flowing along a vector field on Teichm\"uller space that increases the length of the shortest geodesics on a hyperbolic surface. Compare with the construction of Ji--Wolpert \cite[Section 3]{jiwolpert}, also used in \cite{hainautpetersen}. Similarly the deformation retraction onto $\{D_{\mathrm{red}}\}$ is heuristically given by first shortening all geodesics within the curve system $D$, and then lengthening all closed geodesics outside $D$. 
\end{rem}

\begin{proof}[Proof of Theorem~\ref{thm:simplifying}]
We will apply \cref{hocolim1} to $\dRGB$ using the downwards-closed subposets $\mathsf Q_0$ and $\mathsf Q_1$ in \cref{p:colour_forgetting_maps}. From the definition of $\dRGB$ it is clear that $\dRGB =\mathsf Q_0 \cup \mathsf Q_1$, and in Proposition~\ref{p:colour_forgetting_maps} we showed that $\mathsf Q_0$ and $\mathsf Q_1$ are contractible. It remains to show $$|\mathsf Q_0 \cap \mathsf Q_1|\simeq |\NS(V)|.$$
There is a poset map $g:\mathsf Q_0 \cap \mathsf Q_1\to \NS(V)^{\text{op}}$ given by taking $E \in \mathsf Q_0 \cap \mathsf Q_1$ to the disc system $g(E)$ consisting of red and green discs in $E$. Then $g(E)$ is non-simple as $E \in \mathsf Q_0$, and nonempty as $E \in \mathsf Q_1$. Applying the fibre lemma to $g:\mathsf Q_0 \cap \mathsf Q_1\to \NS(V)^{\text{op}}$, it suffices to show that each fibre $g^{-1}(\NS(V)^\text{op}_{\leq D})$ is contractible. A point $E \in \mathsf Q_0 \cap \mathsf Q_1$ lies in $g^{-1}(\NS(V)^\text{op}_{\leq D})$ if and only if the set of green and red discs in $E$ contains $D$. Define a self-map $\phi$ of $g^{-1}(\NS(V)^\text{op}_{\leq D})$ as follows: $\phi(E)$ is given by turning any red disc in $E$ not belonging to $D$ green, followed by turning any green disc of $E$ not belonging to  $D$ blue, followed by turning any green disc of $E$ belonging to $D$ red. This is a deflating map followed by two inflating maps, so is a homotopy equivalence onto its image by \cref{monotone}. Elements of $\mathrm{Im}(\phi)$ are precisely extensions of $D$ (with all discs coloured red) by simple disc systems in $V \setminus D$ (with all discs coloured blue), so that $\mathrm{Im}(\phi) \cong \SD(V\setminus D)$, which is contractible.
\end{proof}

\begin{rem} \label{rem:alternative_contractibility_proof} As we showed that $|\RGB(V)| \simeq \Y(V)$ in \cref{thm:combinatorial model 2}, combining this with \cref{p:colour_forgetting_maps} gives a new proof of contractibility of $\Y(V)$, which is the most substantial part of \cref{hp}.

A key input in the proof of \cref{hp} given in \cite{hainautpetersen} was Giansiracusa's theorem \cite{giansiracusa} that the derived modular envelope of the framed little disc operad is equivalent to the handlebody modular operad. In turn, an important ingredient in Giansiracusa's proof of this theorem is the contractibility of $|\SD(V)|$. Thus in both approaches the result is fundamentally a consequence of contractibility of the complex of simple disc systems. \end{rem}

\section{Proof of the main theorems}

In this section, we prove Theorems \ref{thm:A}, \ref{thm:B}, and \ref{thm:C} from the introduction. To repeat the strategy:  $\hmod(V)$ is a virtually torsion-free group, and we have constructed a contractible manifold $\mathcal{Y}(V)$ of dimension $6g-6+3b+2p$ that admits a proper, cocompact action of $\hmod(V)$ (\cref{hp}). We then showed that $\partial \mathcal{Y}(V)$ is homotopy equivalent to $\Sigma \compNS(V)$ (\cref{thm:combinatorial model 2} and \cref{thm:simplifying}). The criterion in \cref{criterion} then shows that Theorems \ref{thm:A} and \ref{thm:B} are at this point both equivalent to each other, and that both of them imply \cref{thm:C}. 

For parts  of this section it will be convenient to work ``before barycentric subdivision'' --- recall that $\compD(V)$ and $\compNS(V)$ denote simplicial complexes where vertices are given by discs and simplices by disc systems, as defined in connection with \cref{thm:B} of the introduction. The barycentric subdivisions of $\compD$ and $\compNS$ are isomorphic to the geometric realizations of the posets $\mathsf D$ and $\mathsf{NS}$, respectively. 

The first step is high (homological) connectivity of $\compNS$, which will be an easy consequence of the fact that we already know that $\operatorname{vcd}\hmod(V) =d(g,b,p)$ \cite{hirose}, with notation as in \cref{thm:A}. In fact, we will only need the inequality $\operatorname{vcd} \hmod(V)\leq d(g,b,p)$, which follows directly from the fact that $\hmod(V)$ is a subgroup of $\mcg(S)$, and Harer's theorem \cite{harer} that $\operatorname{vcd} \mcg(S)=d(g,b,p).$

\begin{prop}\label{high connectivity} Keep hypotheses and notation as in \cref{thm:B}. Then $$\widetilde H_k(\compNS(V),\Z)=0$$ for all $k<\nu(g,b,p).$
\end{prop}

\begin{proof}Let $G$ be a torsion-free finite index subgroup of $\hmod(V)$. Apply \cref{action} to the action of $G$ on $\Y(V)$. As the action of $G$ on $\Y(V)$ is free and cocompact, $H^*_c(\Y(V),\mathbb{Z})\cong H^*(G,\mathbb{Z}[G])$. Since $\operatorname{cd}(G)=d(g,b,p)$, we see that $H^k_c(\Y(V),\mathbb{Z})=0$ for $k>d(g,b,p)$. Lefschetz duality then tells us that $\widetilde H_k(\partial \Y(V),\Z)=0$ for $$k<\dim(\Y(V))-d(g,b,p)-1.$$
Now we also have $\widetilde H_k(\partial \Y(V),\Z) \cong \widetilde H_{k-1}(\compNS(V),\Z)$ by \cref{thm:simplifying}. This finishes the proof, since $\dim(\Y(V)) - d(g,b,p)-2=\nu(g,b,p).$
\end{proof}

\begin{defn}Let $X$ be a topological space. We say that $X$ is \emph{homologically of dimension $\leq d$} if $H_k(X,\Z)$ vanishes for $k>d$, and $H_d(X,\Z)$ is a free abelian group.\end{defn}

\begin{rem}
    A CW complex of dimension $\leq d$ is homologically of dimension $\leq d$. 
\end{rem}

The second step, then, will be to show that $\compNS(V) \cong |\NS(V)|$ is homologically of dimension $\leq \nu(g,b,p)$. We do this using the following \cref{dimension lemma}. (Similar arguments were used by Harer \cite{harer} and Ivanov \cite{ivanov}.) 

\begin{lem}\label{MV lemma}Let $\mathsf Q$ be a poset, and $x \in \mathsf Q$ a minimal element.  If $|\mathsf Q \setminus \{x\}|$ is homologically of dimension $\leq d$, and $|\mathsf Q_{>x}|$ is homologically of dimension $\leq (d-1)$, for some integer $d \geq 0$, then $|\mathsf Q|$ is homologically of dimension $\leq d$. Moreover, the map $H_d(|\mathsf Q \setminus \{x\}|,\Z)\to H_d(|\mathsf Q|,\Z)$ is a split injection.
\end{lem}

\begin{proof}
    Let $A = |\mathsf Q \setminus \{x\}|$, $B = |\mathsf Q_{>x}|$, and $C = |\mathsf Q_{\geq x}|$. Note that $C$ is a cone over $B$. The result is a consequence of the Mayer--Vietoris sequence:
 \[ \dots
 \to H_p(B,\Z) \to H_p(A,\Z) \oplus H_p(C,\Z) \to H_p(|\mathsf Q|,\Z) \to H_{p-1}(B,\Z) \to \dots \]The final statement, i.e.~that the map $H_d(A,\Z)\to H_d(|\mathsf Q|,\Z)$ is a split injection, follows since its cokernel is a subgroup of $H_{d-1}(B,\Z)$ and hence free abelian.\end{proof}

\begin{defn}We define the \emph{height} of a poset as the supremum of the heights of all its elements (\cref{ht of element}). Equivalently, the height of $\Po$ is the dimension of $|\Po|$.
\end{defn}

\begin{lem}\label{dimension lemma}
	Let $\Po$ be a countable poset of finite height. Assume for all $x \in \Po$ that $|\Po_{>x}|$ is homologically of dimension $\leq (d-1)$, for some $d \geq 0$. Then $|\Po|$ is homologically of dimension $\leq d$. 
\end{lem}

\begin{proof}We prove the lemma by induction on the height of $\Po$. The base case is that $\Po$ has height $0$, in which case the result is obvious, since $|\Po|$ is discrete. So assume that $\Po$ has finite height and that the result is known for all posets of strictly smaller height. Let $\Po_{\min}$ be the subset of minimal elements of $\Po$, and $\Po' = \Po \setminus \Po_{\min}$. We enumerate the elements of $\mathsf P_{\min}$ by $x_1,x_2,\ldots$ and define for $k \geq 0$ \[\mathsf P_k = \mathsf P' \cup \{x_1,\dots,x_k\}.\] We now prove, by induction on $k$, that $|\Po_k|$ is homologically of dimension $\leq d$ for all $k$. Since $\Po_0=\Po'$ has strictly smaller height than $\Po$, it is homologically of dimension $\leq d$ by our induction on height, which proves the base case. The induction step uses \cref{MV lemma} with $\mathsf Q= \Po_k$ and $x=x_k$ (and hence $\mathsf Q \setminus \{x\} = \Po_{k-1}$). 

Finally we claim that also $|\Po| = \varinjlim_k |\Po_k|$ is homologically of dimension $\leq d$. The only thing which is  not clear is that $H_d(|\Po|,\Z)$ is free abelian, as a filtered colimit of free abelian groups is in general only torsion-free. But in this case we are taking a sequential colimit of split injections, by \cref{MV lemma}. Then the colimit, too, is free: indeed, if 
\[ 0 = M_0 \to M_1 \to M_2 \to \dots \]
is such a directed system, and $N_k$ denotes a choice of complement of $M_{k-1}$ in $M_k$ for all $k$, then $\varinjlim_k M_k \cong \bigoplus_{k=1}^\infty N_k$. \end{proof}

\begin{notation}\label{notation:join}If $\Po$ is a poset, let $\mathsf C \Po = \Po \cup \{0\}$ denote $\Po$ adjoined a minimum element. If $\mathsf P$ and $\mathsf Q$ are posets, let $\mathsf P \altjoin \mathsf Q := (\mathsf{CP} \times \mathsf{CQ}) \setminus \{(0,0)\}$, where the poset structure on $\mathsf{CP} \times \mathsf{CQ}$ is componentwise. 
\end{notation}

\begin{rem}
    The operation $\altjoin$ is associative. There is a natural homeomorphism $|\mathsf P \altjoin \mathsf Q| \cong |\mathsf P| \ast |\mathsf Q|$ on geometric realizations \cite[Proposition 1.9]{quillen-homotopyproperties}. Hence the operation $\altjoin$ provides another lift of the join operation from spaces to posets than the one of \cref{notation:quillenjoin}. 
\end{rem}

\begin{prop}\label{final prop}Let $V$ be a handlebody of genus $g \geq 2$, with no marked points or discs on its boundary. Assume that the statement of \cref{thm:B} is known for all handlebodies of genus strictly less than $g$, with any number of marked points or boundary components. Then  the statement of \cref{thm:B} holds also for $V$; that is, $|\NS(V)|$ is homology equivalent to a wedge of spheres of dimension $2g-3$. 
\end{prop}

\begin{proof}It suffices to prove that $|\NS(V)|$ is homologically of dimension~$\leq (2g-3)$, after \cref{high connectivity}. We prove this using the criterion of \cref{dimension lemma}. Thus we choose $D \in \NS(V)$ and prove that $|\NS(V)_{>D}|$ is homologically of dimension $\leq (2g-4)$.

Suppose that $D$ cuts $V$ into handlebodies $V_1,\dots,V_{s}$ of positive genera $g_1,\dots,g_s$,  each of which has $b_1,\dots,b_s$ marked boundary disks, and $W_1,\dots,W_t$ of genus zero, each of which has $c_1,\dots,c_t$ marked boundary disks. Observe the relation
\begin{equation}\tag{$\star$}\label{eq:euler char}
    2g-2 = \sum_{i=1}^s (2g_i - 2 + b_i) + \sum_{j=1}^t (-2+c_j).
\end{equation} 
Using \cref{notation:join} we now have the formula
$$ \NS(V)_{>D} \cong \mathsf X \altjoin \D(W_1) \altjoin \dots \altjoin \D(W_t)$$
where $\mathsf X$ denotes the subposet of $\D(V_1) \altjoin \dots \altjoin \D(V_s)$ given by the union
\[
 \bigcup_{i=1}^s \D(V_1) \altjoin \dots \altjoin \D(V_{i-1}) \altjoin \NS(V_i) \altjoin \D(V_{i+1}) \altjoin \dots \altjoin \D(V_s).
\]Indeed, this expresses that a non-simple disc system properly containing $D$ will be given by a nonempty family of discs in $V \setminus D$, such that for at least one $i=1,\dots,s$, those discs lying in $V_i$ will form a non-simple system (possibly empty), whereas the discs in the remaining components may be arbitrary. 

Now each  disc complex is contractible by \cref{disk cx contractible}. Hence \cref{hocolim2} applies, with $\mathsf P = \mathsf X$ and $\mathsf Q_i = \D(V_1) \altjoin \dots \altjoin \D(V_{i-1}) \altjoin \NS(V_i) \altjoin \D(V_{i+1}) \altjoin \dots \altjoin \D(V_s)$, to show that $$|\mathsf X| \simeq \Sigma^{s-1} (|\NS(V_1)| \ast \dots \ast |\NS(V_s)|) .$$ Since $g_i<g$ for all $i$, we know by induction on genus that each factor $|\NS(V_i)|$ is homology equivalent to a wedge of spheres of dimension $2g_i-4+b_i$. Since the disc complex and the curve complex coincide in genus zero, $|\D(W_j)|$ is homotopic to a wedge of spheres of dimension $-4+c_j$ by Harer \cite{harer}. Thus $|\NS(V)_{>D}|$ is homology equivalent to a wedge of spheres of dimension 
\[ t+ 2s-2+ \sum_{i=1}^s (2g_i-4+b_i) + \sum_{j=1}^t (-4+c_j) = 2g-4-t,\]
using Equation \eqref{eq:euler char}. The result follows. \end{proof}

\begin{prop}\label{birman trick} Fix $g \geq 2$. If \cref{thm:A} holds when $b=p=0$, then it holds for all values of  $b$ and $p$. Similarly when $g=1$, if \cref{thm:A} holds when $b=0$ and $p=1$, then it holds for all values of $b$ and $p$ with $b+p>0$. 
\end{prop}

\begin{proof}
    This is an application of \cref{prop:SES}. Firstly, if $V'$ is obtained from $V$ by replacing one of the boundary discs with a marked point, then there is a short exact sequence $$ 0 \to \Z \to \hmod(V) \to \hmod(V')\to 1,$$
    so \cref{prop:SES} implies that $\hmod(V)$ satisfies \cref{thm:A} whenever $\hmod(V')$ does. Hence \cref{thm:A} holds in general, if it holds in the special case that $b=0$.
    
    Secondly, let $V''$ be obtained from $V$ by forgetting one of the marked boundary points, and let $S''$ be the boundary of $V''$, minus its marked points and boundary discs. If $S''$ has negative Euler characteristic, then by \cite[Lemma~3.3]{hensel} there is instead a short exact sequence (a Birman exact sequence for handlebodies)
    $$ 1 \to \pi_1(S'') \to \hmod(V) \to \hmod(V'')\to 1.$$
    Again $\hmod(V)$ satisfies \cref{thm:A} whenever $\hmod(V'')$ does. Hence \cref{thm:A} holds in general, if it holds in the special case that $p=0$ (or $p=1$ when $g=1$).     
\end{proof}

\begin{prop}\label{genus one}
    The statement of \cref{thm:A} holds when $g=1$. 
\end{prop}

\begin{proof}
    By \cref{birman trick}, it suffices to show this when $(g,b,p)=(1,0,1)$. But in this case we have $\hmod(V)\cong \Z \rtimes \{\pm 1\}$, which is clearly a virtual duality group.
\end{proof}

\begin{thm}
    Theorems \ref{thm:A}, \ref{thm:B}, and \ref{thm:C} hold.
\end{thm}
\begin{proof}
    As observed in the first paragraph of this section, we have already proven that the statements of \cref{thm:A} and \cref{thm:B} are equivalent when $g>0$, and that both of them imply \cref{thm:C}. (When $g=0$, \cref{thm:A} is a theorem of Harer and Theorems \ref{thm:B} and \ref{thm:C} are vacuous.)

    Now \cref{genus one} shows that \cref{thm:A} holds when $g=1$. Thus we may fix some $g \geq 2$. According to \cref{birman trick}, to show the statement of \cref{thm:B} for \emph{all} values of $b$ and $p$ in this genus, it suffices to show it for 
    $b=p=0$. But by induction on the genus we may assume the statement to be known for any handlebody of lower genus, so we are done by \cref{final prop}. 
\end{proof}

\begin{rem}The precise inductive strategy is quite important. Specifically, it is crucial that \cref{dimension lemma} is only ever applied to the case $b=p=0$ of no marked discs or points on the boundary (as in \cref{final prop}), and that the case of markings is instead dealt with using the ``Birman exact sequence'' via \cref{birman trick}. 

Indeed, in order to apply \cref{dimension lemma} to $\NS(V)$, we need that $\NS(V)_{>D}$ has strictly lower homological dimension than $\NS(V)$, for all $D \in \NS(V)$. This is only true in the closed case, the point being that the homological dimension of $|\NS(V)|$ is $-\chi(S)-1$ when $b=p=0$, and $-\chi(S)-2$ when $b+p>0$. In this sense the dimension is one higher than expected when $b=p=0$. 
Note also that in applying \cref{dimension lemma} we cut the handlebody into pieces along some nonempty disc system, in which case all pieces have at least one marked disc on its boundary.  
\end{rem}

Other than in the low genus cases, one can show the stronger result that $\compNS(V)$ is \emph{homotopy equivalent} to a wedge of spheres.

\begin{prop} \label{prop: simple connectivity}
If $g \geq 3$, then $\compNS(V)$ is homotopy equivalent to a wedge of spheres of dimension $\nu(g,b,p)$, where $\nu(g,0,0) = 2g-3$ and $\nu(g,b,p) = 2g-4+b+p$ if $b+p>0$.
\end{prop}

\begin{proof}
After \cref{thm:B}, it is enough to show that $\compNS(V)$ is  simply connected when $g \geq 3$. By the work of McCullough \cite{mccullough} and Hatcher--Wahl \cite{hatcherwahl} the full disc complex $\compD(V)$ is contractible (see Theorem~\ref{disk cx contractible}): we will show that the inclusion $\compNS(V) \subset \compD(V)$ is 1-connected.  

As a simple disc system needs to use at least $g$ discs, it follows that $\compNS(V)$ contains the $(g-2)$-skeleton of $\compD(V)$, and so contains the $2$-skeleton of $\compD(V)$ when $g \geq 4$. Thus, $\compNS(V)$ is simply connected for $g \geq 4$, as $\compD(V)$ is.

When $g=3$, the complex $\compNS(V)$ contains the $1$-skeleton of $\compD(V)$. We will show that the boundary $\partial \sigma$ of every 2-cell $\sigma \in \compD(V) \setminus \compNS(V)$  is still bounded by a disc in $\compNS(V)$, which implies that $\compNS(V)$ is also simply connected, as above. When $g=3$, the only 2-cells in $\compD(V)$ forming simple disc systems consist of three nonseparating discs.  We can obtain a 3-cell $\sigma'$ of $\compD(V)$ by adding an arbitrary disjoint separating disc to $\sigma$. Every face of $\sigma'$ other than $\sigma$ contains this separating disc, and is therefore contained in $\compNS(V)$, so that the remaining faces of $\sigma'$ form a disc bounding $\partial \sigma$. It follows that the inclusion $\compNS(V) \subset \compD(V)$ is 1-connected, as required.
\end{proof}

\begin{figure}
\begin{tikzpicture}
\draw[smooth] (0,1) to[out=30,in=150] (2,1) to[out=-30,in=210] (3,1) to[out=30,in=150] (5,1) to[out=-30,in=210] (6,1) to[out=30,in=150] (8,1) to[out=-30,in=30] (8,-1) to[out=210,in=-30] (6,-1) to[out=150,in=30] (5,-1) to[out=210,in=-30] (3,-1) to[out=150,in=30] (2,-1) to[out=210,in=-30] (0,-1) to[out=150,in=-150] (0,1);
\draw[smooth] (0.4,0.1) .. controls (0.8,-0.25) and (1.2,-0.25) .. (1.6,0.1);
\draw[smooth] (0.5,0) .. controls (0.8,0.2) and (1.2,0.2) .. (1.5,0);
\draw[smooth] (3.4,0.1) .. controls (3.8,-0.25) and (4.2,-0.25) .. (4.6,0.1);
\draw[smooth] (3.5,0) .. controls (3.8,0.2) and (4.2,0.2) .. (4.5,0);
\draw[smooth] (6.4,0.1) .. controls (6.8,-0.25) and (7.2,-0.25) .. (7.6,0.1);
\draw[smooth] (6.5,0) .. controls (6.8,0.2) and (7.2,0.2) .. (7.5,0);
\draw[thick, purple] (-0.5,0) arc(180:0:0.51 and 0.2);
\draw[thick, purple, dashed] (-0.5,0) arc(180:0:0.51 and -0.2);
\draw[thick, orange] (2.5,-0.85) arc(270:90:0.3 and 0.85);
\draw[thick, orange, dashed] (2.5,-0.85) arc(270:450:0.3 and 0.85);
\draw[thick, purple] (4.0,-0.15) arc(270:90:0.3 and -1.14/2);
\draw[thick, purple,dashed] (4.0,-0.15) arc(270:450:0.3 and -1.14/2);
\draw[thick, purple] (8.5,0) arc(180:0:-0.51 and 0.2);
\draw[dashed, thick, purple] (8.5,0) arc(180:0:-0.51 and -0.2);

\coordinate (A) at (10,0); 
\coordinate (B) at (12,0); 
\coordinate (C) at (11,1); 
\coordinate (D) at (11,-1); 

\draw[thick, dashed] (A) -- (B);
\draw[thick] (A) -- (C);
\draw[thick] (A) -- (D);
\draw[thick, purple] (B) -- (C);
\draw[thick, purple] (C) -- (D);
\draw[thick, purple] (D) -- (B);

\fill[orange] (A) circle (2pt); 
\fill[purple] (B) circle (2pt); 
\fill[purple] (C) circle (2pt); 
\fill[purple] (D) circle (2pt); 
\end{tikzpicture}
\caption{Four discs in a genus three handlebody $V$ (left) and the corresponding simplex in $\mathbf D(V)$ (right). The three purple discs span a $2$-cell $\sigma \in \mathbf D(V) \setminus \mathbf{NS}(V)$; the cone on $\partial \sigma$ with respect to the orange vertex lies in $\compNS(V)$.}    
\end{figure}

\begin{rem}
    The argument in \cref{prop: simple connectivity} may be considered as an instance of the \emph{bad simplex argument} \cite[Corollary 2.2(a)]{hatchervogtmann} applied to the pair $(\compD(V),\compNS(V))$, taking all simplices of $\compD(V) \setminus \compNS(V)$ to be bad. To prove that $\compNS(V)$ is $1$-connected, one only needs to study bad simplices of dimension $\leq 2$. If one wanted to push the argument further, then in genus $2$ one would need to consider both bad $2$-cells and bad edges, and in genus $1$ one would moreover need to consider bad vertices. Bad $2$-cells could be dealt with in the same manner in any genus, but bad edges and vertices seem like they would moreover require connectivity or simple connectivity of complexes of separating curves on a partitioned genus zero surface (cf.\ \cite{looijenga}). \end{rem}

\bibliographystyle{alpha}
\bibliography{database}

\end{document}